\documentclass[12pt,reqno]{amsart}
\usepackage{stmaryrd}
\usepackage{amssymb}
\usepackage{fullpage}
\usepackage{color}
\usepackage{aliascnt}
\usepackage{enumitem}
\usepackage{tikz}
\usetikzlibrary{matrix,arrows,decorations.markings,shapes}
\usepackage{tikz-cd}
\usepackage{algorithm,algpseudocode} 
\usepackage[bookmarksdepth=2,linktoc=page,colorlinks,linkcolor={red!70!black},citecolor={red!80!black},urlcolor={blue!80!black}]{hyperref}
\usepackage{cleveref}

\allowdisplaybreaks 

\theoremstyle{plain}
\newtheorem{thm}{Theorem}[section] 
\newaliascnt{lemma}{thm}\newtheorem{lemma}[lemma]{Lemma}\aliascntresetthe{lemma}
\newaliascnt{cor}{thm}\aliascntresetthe{cor}
\newaliascnt{prop}{thm}\newtheorem{prop}[prop]{Proposition}\aliascntresetthe{prop}

\newtheorem*{thm*}{Theorem}
\newtheorem*{lem*}{Lemma}
\newtheorem*{cor*}{Corollary}

\theoremstyle{definition}
\newaliascnt{definition}{thm}\newtheorem{definition}[definition]{Definition}\aliascntresetthe{definition}
\newaliascnt{rem}{thm}\newtheorem{rem}[rem]{Remark}\aliascntresetthe{rem}
\newaliascnt{ex}{thm}\newtheorem{ex}[ex]{Example}\aliascntresetthe{ex}
\newtheorem{question}[thm]{Question}

\newtheorem*{definition*}{Definition}
\newtheorem*{ex*}{Example}
\newtheorem*{rem*}{Remark}

\theoremstyle{remark}

\DeclareMathOperator{\FE}{FE}
\DeclareMathOperator{\FP}{FP}
\DeclareMathOperator{\Mor}{Mor}
\DeclareMathOperator{\im}{im}
\DeclareMathOperator{\coker}{coker}
\DeclareMathOperator{\Hom}{Hom}
\DeclareMathOperator{\Ext}{Ext}
\DeclareMathOperator{\rk}{rank}
\DeclareMathOperator{\rref}{rref}
\DeclareMathOperator{\sgn}{sgn}
\DeclareMathOperator{\Cr}{Cr}
\DeclareMathOperator{\hypersum}{\,\raisebox{-2.2pt}{\larger[2]{$\boxplus$}}\,}
\DeclareMathOperator{\diag}{diag}

\newcommand\D{{\mathbb D}}
\newcommand\F{{\mathbb F}}
\renewcommand\H{{\mathbb H}}
\newcommand\K{{\mathbb K}}

\newcommand\Q{{\mathbb Q}}
\newcommand\R{{\mathbb R}}
\renewcommand\S{{\mathbb S}}
\newcommand\T{{\mathbb T}}
\newcommand\U{{\mathbb U}}
\newcommand\Z{{\mathbb Z}}

\newcommand{\bfz}{\textbf{0}}

\renewcommand{\emptyset}{\varnothing}

\title{Representing matroids via pasture morphisms}

\author{Justin Chen, Tianyi Zhang}
\address{\rm School of Mathematics, Georgia Institute of Technology, Atlanta, Georgia, USA}
\email{\{jchen646,tzhang358\}@gatech.edu}


\begin{document}

\vspace*{-0.5cm}
\maketitle

\begin{abstract}
Using the framework of pastures and foundations of matroids developed by Baker-Lorscheid, we give algorithms to: (i) compute the foundation of a matroid, and (ii) compute all morphisms between two pastures.
Together, these provide an efficient method of solving many questions of interest in matroid representations, including orientability, non-representability, and computing all representations of a matroid over a finite field.
\end{abstract}

\section{Introduction}
A matroid is a combinatorial structure that abstracts the notion of linear independence in linear algebra.
Given a matrix $A$ (over a field), one obtains a matroid $M[A]$ with ground set equal to the columns of $A$, and bases given by bases of the column space. 
A natural question is: given a matroid $M$, can one find a matrix $A$ such that $M[A] = M$, i.e. is $M$ \emph{representable} by a matrix $A$?
One approach (cf. \cite{lunelli2002representation}) is to take the entries of $A$ to be indeterminates and use the nonbases/bases to set up a system of polynomial equations/inequations.
However, such systems quickly exceed the practical capabilities of general-purpose polynomial solvers (e.g. using Gr\"obner bases), even for matroids of medium size (i.e. $\ge 6$ elements).

In another direction, there exist interesting matroids -- like the Vamos and non-Pappus matroids -- that are not representable over any field, and for such matroids one would still like analogous ``representations".
Baker and Lorscheid \cite{Baker-Lorscheid20} introduced the category of
 \emph{pastures}, which contains the category of fields, and is better behaved categorically (e.g. has all small limits and colimits, see \cite{creech2021limits}).
 They also introduced \emph{representations of matroids over pastures}, which are analogous to matrices in the case of fields.
For instance, both the Vamos and non-Pappus matroids are \emph{orientable}, i.e. are representable over the sign hyperfield.

In \cite{Baker-Lorscheid20}, Baker and Lorscheid also introduced the \emph{foundation} of a matroid $M$, which is a pasture $F_M$ that captures all data of $P$-representations of $M$ for any pasture $P$.
More precisely, for any pasture $P$, there is a bijection between rescaling equivalence classes of $P$-representations of $M$ and pasture morphisms $F_M \to P$, see \Cref{thm:UP}.
This motivates the following computational problems for a matroid $M$:
\begin{enumerate}
    \item Compute the foundation $F_M$ of $M$,
    \item Given a pasture $P$, compute all pasture morphisms $F_M \to P$,
    \item Given a pasture $P$ and a pasture morphism $\varphi : F_M \to P$, compute a $P$-representation of $M$ corresponding to $\varphi$.
    In particular, when $P$ is a field, compute a matrix corresponding to $\varphi$.
\end{enumerate}
Our goal is to develop algorithms to solve these 3 problems.
We recall some background on the category of pastures in \Cref{sec:pastures}.
Next, we discuss foundations in \Cref{sec:foundationAlgorithm}, and give an algorithm to compute the foundation of a matroid.
In \Cref{sec:morphismAlgorithm}, we explain in detail our algorithm to compute pasture morphisms.
Finally, we discuss some applications in \Cref{sec:matrixReps}, including how to obtain a matrix representation from a given pasture morphism.

\section{Pastures} \label{sec:pastures}

\subsection{The category of pastures} \label{ssec:pastureCategory}
We begin with the category of pastures (cf. \cite[p. 4]{Baker-Lorscheid20}):

\begin{definition}
A \emph{pasture} is a monoid $(P, \cdot)$ with identity $1$ and absorbing element $0$, such that $P^\times := (P \setminus \{0\}, \cdot)$ is an abelian group, together with a \emph{nullset} $N_P \subseteq P^3$ satisfying:
\begin{enumerate}
    \item [(P1)] $(1,0,0) \not \in N_P$
    \item [(P2)] There is a unique element $\epsilon \in P$ such that $(1,\epsilon,0) \in N_P$
    \item [(P3)] $N_P$ is closed under the action of $S_3 \times P$: $(\sigma, p) \cdot (a_1, a_2, a_3) := (a_{\sigma(1)} p, a_{\sigma(2)} p, a_{\sigma(3)} p)$.
\end{enumerate}
\end{definition}

We denote by ``$a + b + c = 0$'' the $S_3$-orbit of a triple $(a,b,c) \in N_P$, and view this as a 3-term \emph{additive relation} in $P$.
If $c = 0$, we write ``$a + b = 0$'' as a 2-term additive relation.

\begin{definition} \label{def:pastureMorphism}
Let $P_1, P_2$ be pastures.
A {\em pasture morphism} is a (set) map $f : P_1 \to P_2$ such that $f(0) = 0$, $f$ restricts to a group homomorphism $P_1^\times \to P_2^\times$, and $f^3 : P_1^3 \to P_2^3$ restricts to a map $N_{P_1} \to N_{P_2}$.
A pasture morphism $f: P_1 \to P_2$ is an \emph{isomorphism} if there exists a pasture morphism $g: P_2 \to P_1$ such that $g \circ f = \text{id}_{P_1}$ and $f \circ g = \text{id}_{P_2}$.
\end{definition}

In general, a pasture may have (infinitely) many additive relations.
However, by axiom (P3), we may always rescale one of the nonzero terms to be $\epsilon$.
Then by axiom (P2) (along with $\epsilon^2 = 1$), all 2-term relations are $P^\times$-multiples of $1 + \epsilon = 0$.
The 3-term relations require more consideration -- assuming $c \ne 0$, there are 3 ways to rescale a 3-term relation to make one term equal to $\epsilon$:
\[
    a+b+c = 0 \iff x + y + \epsilon = 0, \;\; \epsilon + \frac{1}{x} + \epsilon \cdot \frac{y}{x} = 0, \;\; \epsilon \cdot \frac{x}{y} + \epsilon + \frac{1}{y} = 0
\]
where $x = \epsilon \cdot \frac{a}{c}$ and $y = \epsilon \cdot \frac{b}{c}$.
This motivates the following definitions (cf. \cite[Section 3.2]{Baker-Lorscheid18}):

\begin{definition} \label{def:fundDefs}
Let $P$ be a pasture.
An element $x \in P^\times$ is called a \emph{fundamental element} if there exists $y \in P^\times$ such that $x + y + \epsilon = 0$, i.e. ``$x + y = 1$''.
In this case we say that $y$ is a \emph{partner} of $x$, and call the ordered pair $(x, y)$ a \emph{fundamental pair}. 
We denote by $\FE(P)$ resp. $\FP(P)$ the set of fundamental elements resp. pairs of $P$.

There is a natural equivalence relation on $\FP(P)$: two fundamental pairs are equivalent if they give rise to the same 3-term relation (up to rescaling by $P^\times$).
We define a \emph{hexagon} to be an equivalence class, and represent it by a triple $((x, y), (\frac{1}{x}, \epsilon \cdot \frac{y}{x}), (\frac{1}{y}, \epsilon \cdot \frac{x}{y}))$ for some $(x, y) \in \FP(P)$.
Thus for any fundamental pair $(x, y)$, exactly one of $(x, y)$ and $(y, x)$ appears in a unique hexagon.
\end{definition}

We next introduce some terminology for conditions on pastures:

\begin{definition} \label{def:pastureConditions}
Let $P$ be a pasture.
We say that:
\begin{itemize}
\item $P$ is \emph{slim} if every fundamental element has a unique partner
\item $P$ is \emph{finitely generated} if $P^\times$ is finitely generated as an abelian group
\item $P$ is \emph{finitely related} if $|\FP(P)| < \infty$
\item $P$ is \emph{finitely presented} if $P$ is both finitely generated and finitely related.
\end{itemize}
\end{definition}

Many pastures of interest (including all the ones considered in this paper) are finitely presented, e.g. foundations of finite matroids, finite fields, and many other interesting partial fields and hyperfields.

We can now reformulate \Cref{def:pastureMorphism} in terms of fundamental pairs (note that 2-term relations are automatically handled by the condition $f (\epsilon_1) = \epsilon_2$):

\begin{lemma} \label{lem:Hom}
Let $P_1, P_2$ be pastures.
A set map $f : P_1 \to P_2$ is a morphism of pastures if and only if $f(0) = 0$, $f(\epsilon_1) = \epsilon_2$, $f$ induces a group homomorphism from $P_1^\times$ to $P_2^\times$, and $(f(x), f(y)) \in \FP(P_2)$ for any $(x, y) \in \FP(P_1)$.
\end{lemma}

\begin{lemma} \label{lem:iso}
Let $P_1, P_2$ be pastures with $P_1$ finitely related.
The following are equivalent for a pasture morphism $f : P_1 \to P_2$:
\begin{enumerate}
    \item $f$ is an isomorphism
    \item $f$ induces a group isomorphism $P_1^\times \xrightarrow{\sim} P_2^\times$ and a bijection $N_{P_1} \xrightarrow{\sim} N_{P_{2}}$
    \item $f$ induces a group isomorphism $P_1^\times \xrightarrow{\sim} P_2^\times$ and $|\FP(P_1)| = |\FP(P_2)|$.
\end{enumerate}
\end{lemma}

\begin{proof}
$(1) \implies (2) \implies (3)$ is clear.
For $(3) \implies (1)$: note that by \Cref{lem:Hom}, $f$ induces an injection $\FP(P_1) \hookrightarrow \FP(P_2)$, which is a bijection since $|\FP(P_2)| = |\FP(P_1)|$ is finite.
Let $g^\times : P_2^\times \to P_1^\times$ be the inverse of the induced isomorphism $P_1^\times \xrightarrow{\sim} P_2^\times$, and extend $g^\times$ to $g : P_2 \to P_1$ by setting $g(0) := 0$.
Applying \Cref{lem:Hom} again shows that $g$ is a pasture morphism, which is a two-sided inverse of $f$ by construction.
\end{proof}

\begin{ex} \label{ex:universalObjects}
The category of pastures admits an initial object and a terminal object.
The initial object is the \emph{regular partial field} $\F_1^\pm$, where $\F_1^\pm = \{0,1,-1\}$ as a set, with usual multiplication, $\epsilon = -1$, and no fundamental elements.
The terminal object is the \emph{Krasner hyperfield} $\K$, where $\K = \{0,1\}$ as a set, with usual multiplication, $\epsilon = 1$, and one hexagon $( ( 1,1 ), ( 1,1 ), ( 1,1 ) )$.
\end{ex}

More generally, pastures generalize both partial fields and hyperfields:

\begin{ex} \label{ex:partialHyperfieldsAsPastures}
A partial field $(R, G)$ can be viewed as a pasture $P_G$, by taking $P_G^\times := G$, $\epsilon := -1 \in R$ and declaring that $\alpha + \beta + \gamma = 0$ in $P_G \iff \alpha+\beta+\gamma = 0 \in R$ for all $\alpha, \beta, \gamma \in G$.
Note that every partial field is slim.

A hyperfield $H$ can be viewed as a pasture $P_H$, by taking $P_H^\times := H^\times$, $\epsilon := -1 \in H$ and declaring that $\alpha + \beta + \gamma = 0$ in $P_H \iff 0 \in \alpha \hypersum \beta \hypersum \gamma \subseteq H$ for all $\alpha, \beta, \gamma \in H^\times$.
For instance, the \emph{sign hyperfield} $\S = \R/\R_{>0}$ has multiplicative group $\Z/2\Z$, with one hexagon $( (1, 1), (1, \epsilon ), (1, \epsilon ) )$ (which is not slim).

In particular, fields are both partial fields and hyperfields, and induce the same pasture in both ways.
\end{ex}

\subsection{Representations of matroids over pastures} \label{ssec:pastureReps}

The most interesting pastures for our purposes arise in connection with matroids.
We assume basic familiarity with matroids, and refer to \cite{Oxley92} for matroid-theoretic terminology.
Let $[n] := \{ 0, \ldots, n-1 \}$ denote a totally ordered $n$-element set (serving as the ground set of a matroid), and let $\binom{[n]}{k}$ denote the set of all subsets of $[n]$ of size $k$.
We write $\sgn : S_n \to \Z/2\Z = \{0,1\}$ for the sign map.

\begin{definition}[{\cite[p. 5]{Baker-Lorscheid20}}] \label{def:GPfunction}
Let $M$ be a matroid of rank $r$ on $[n]$, and let $P$ be a pasture. 
A {\em $P$-representation of $M$} is a function $\Delta: [n]^r \to P$ such that:
\begin{enumerate}
    \item [(GP1)] $\Delta(x_1, \ldots, x_r) \neq 0$ if and only if $\{ x_1, \ldots, x_r \}$ is a basis of $M$
    \item [(GP2)] $\Delta(\sigma(x_1), \ldots, \sigma(x_r)) = \epsilon^{\sgn(\sigma)} \Delta(x_1, \ldots, x_r)$ for any $\sigma \in S_r$
    \item [(GP3)] For any $J := \{ y_1, \ldots, y_{r-2} \} \in [n]^{r-2}$ and all $\{ e_1,e_2,e_3,e_4 \} \in [n]^4$,
    \[
    \Delta(Je_1e_2)\Delta(Je_3e_4) + \epsilon \Delta(Je_1e_3)\Delta(Je_2e_4) + \Delta(Je_1e_4)\Delta(Je_2e_3) = 0
    \]
    where $Je_ie_j := \{ y_1,\ldots, y_{r-2}, e_i, e_j \}$.
\end{enumerate}
\end{definition}

\begin{rem}
If $P = \K$ is the Krasner hyperfield, then (GP3) over $\K$ reduces to the usual basis exchange axiom, so every matroid has a canonical (in fact, unique) $\K$-representation.
\end{rem}


Next, we recall 2 important types of equivalences among representations over pastures:

\begin{definition}[{\cite[p. 5]{Baker-Lorscheid20}}]
Let $\Delta, \Delta'$ be $P$-representations of a matroid $M$.
\begin{enumerate}
\item We say that $\Delta$ and $\Delta'$ are {\em isomorphic} if there exists $c \in P^\times$ such that $\Delta'(x_1,\ldots,x_r) = c \Delta(x_1,\ldots,x_r)$ for all $\{x_1,\ldots,x_r\} \in [n]^r$. We denote by $\mathcal X^I_M(P)$ the set of isomorphism classes of $P$-representations of $M$.
\item We say that $\Delta$ and $\Delta'$ are {\em rescaling equivalent} if there exists $c \in P^\times$ and a map $d: [n] \to P^\times$ such that $\Delta'(x_1,\ldots,x_r) = c\cdot d(x_1)\cdots d(x_r)\cdot \Delta(x_1,\ldots,x_r)$ for all $\{x_1,\ldots,x_r\} \in [n]^r$. We denote by $\mathcal X^R_M(P)$ the set of rescaling equivalence classes of $P$-representations of $M$.
\end{enumerate}
\end{definition}

\begin{rem}[{\cite[p. 6]{Baker-Lorscheid20}}] \label{rem:functoriality}
For a fixed matroid $M$, the associations $P \mapsto \mathcal X^R_M(P)$ and $P \mapsto \mathcal X^I_M(P)$ are functorial.
In particular, if $f : P_1 \to P_2$ is a pasture morphism, there are natural maps $\mathcal X^R_M(P_1) \to \mathcal X^R_M(P_2)$ and $\mathcal X^I_M(P_1) \to \mathcal X^I_M(P_2)$, called the \emph{pushforward} along $f$. 
\end{rem}

\begin{ex}\label{ex:rescalingEquivalenceClassAndIsomorphismClass}
Let $M$ be a matroid, and $P$ a (partial) field.
For any $P$-representation $\Delta$ of $M$, there exists an $r \times n$ matrix $A$ over $P$ such that for any $\{x_1,\ldots,x_r\} \in [n]^r$, the $r \times r$ minor of $A$ with columns $\{x_1, \ldots, x_r \}$ equals $\Delta(x_1,\ldots,x_r)$.
Now suppose $\Delta, \Delta', \Delta''$ are $P$-representations of $M$ with corresponding matrices $A, A', A''$, such that $\Delta$ is isomorphic to $\Delta'$ and $\Delta$ is rescaling equivalent to $\Delta''$.
We can compute the reduced row echelon forms $\rref(A),\rref(A')$ and $\rref(A'')$ with respect to a basis $B_0$ of $M$ (i.e. fixing the submatrix with columns in $B_0$ to be the identity).
Then $\rref(A) = \rref(A')$, and $\rref(A'')$ can be obtained from $\rref(A)$ by column rescaling.
\end{ex}

The following theorem is a key motivation for this paper:

\begin{thm} \label{thm:UP}
Let $M$ be a matroid and $P$ a pasture.
\begin{enumerate}
\item{\cite[Prop 6.23]{BAKER2021107883}} There exists a (unique) pasture $P_M$, called the \emph{universal pasture} of $M$, representing the functor $P \mapsto \mathcal X^I_M(P)$, i.e. there is a bijection:
\[
 \mathcal X^I_M(P) \xrightarrow{\sim} \Mor(P_M, P).
\]
\item{\cite[Corollary 7.28]{BAKER2021107883}} There exists a (unique) pasture $F_M$, called the \emph{foundation} of $M$, representing the functor $P \mapsto \mathcal X^R_M(P)$, i.e. there is a bijection:
\[
 \mathcal X^R_M(P) \xrightarrow{\sim} \Mor(F_M, P).
\]
\end{enumerate}
\end{thm}

Thus the foundation captures all the representability data of a matroid, over all pastures.
For purposes of computation, it is a nontrivial result that the foundation admits an explicit presentation via generators and relations, which we now  discuss.

\section{Foundation Algorithm} \label{sec:foundationAlgorithm}

In \cite{BAKER2021107883}, Baker and Lorscheid showed that the multiplicative group of the universal pasture (resp. foundation) is isomorphic to the Tutte group (resp. inner Tutte group) \cite[Theorem 6.27 and Corollary 7.13]{BAKER2021107883}.
We first recall the notions of the (inner) Tutte group(s), 
which were introduced by Dress and Wenzel \cite{Dress1989GeometricAF}.
Then we give explicit presentations for the Tutte and inner Tutte groups,
from which we obtain the unit epsilon and the fundamental pairs for the foundation.
We remark that although our presentation is new, justification for most of the content in this section can be found in existing literature.

\begin{definition} \label{def:TutteGroup}
Let $M$ be a matroid of rank $r$ on $[n]$ and $\mathcal{B}$ the set of bases of $M$.
Let $G$ be the free abelian group (written multiplicatively) on the symbols $\{ \epsilon, X_{B} : B\in \mathcal{B} \}$.
For any (independent) set $I \in \binom{[n]}{r-2}$ and $k_1, \ldots, k_4 \in [n]$, define the \emph{cross-ratio}
\[
\Cr(I, k_1, k_2, k_3, k_4) := \epsilon^{\sgn(\sigma_{13}\sigma_{24}\sigma_{14}\sigma_{23})}{X}_{I \cup \{k_1,k_3\}}{X}_{I \cup \{k_2,k_4\}}{X}_{I \cup \{k_1,k_4\}}^{-1}{X}_{I \cup \{k_2,k_3\}}^{-1}
\label{eq:crossratio}
\]
whenever such an expression makes sense (i.e. each subscript is in $\mathcal{B}$), where $\sigma_{ij} \in S_n$ is the permutation bringing $I \cup \{k_i,k_j\}$ into increasing order.
Fix a basis $B_0 \in \mathcal B$, and set
\[
R(M, B_0) := \langle \epsilon^2, X_{B_0}, \Cr(I, k_1, k_2, k_3, k_4) : I \in \textstyle{\binom{[n]}{r-2}}, k_i \in [n], I \cup \{ k_1, k_2 \} \not \in \mathcal{B} \rangle \subseteq G.
\]
Finally, define the \emph{Tutte group} $\T = \T_M := G/R(M, B_0)$.
\end{definition}

We note that this definition of the Tutte group is slightly different than what appears in e.g. \cite[Definition 6.26]{BAKER2021107883}.
The choices made here necessitate the introduction of sign terms in the cross-ratios, but in return reduce the number of generators to $|\mathcal B| + 1$ (as opposed to $n^r + 1$ or $\binom{n}{r} + 1$), and also realize the Tutte group as a quotient (as opposed to a subquotient) of a free group.
Up to isomorphism, $\T_M$ is independent of the choice of basis $B_0 \in \mathcal B$.

Concretely, we identify the ambient free abelian group $G$ with $\Z^{(|\mathcal B|)+1}$, which has a $\Z$-basis $\{ e_\epsilon, e_B : B \in \mathcal B \}$.
\Cref{alg:TutteGroup} constructs a presentation matrix $R_1$ of $\T$, i.e. a matrix over $\Z$ with $\T = \coker R_1$.
(Hereafter we use $\mid$ to denote concatenation of matrices.)

\begin{algorithm}
\caption{Relations for Tutte group}
\begin{algorithmic}[1]
\Require A matroid $M = ([n], \mathcal B)$ of rank $r$ with a fixed basis $B_0 \in \mathcal B$
\Ensure A matrix $R_1$ over $\Z$ with $|\mathcal B|+1$ rows whose columns generate $R(M, B_0)$
\Procedure{TutteGroup}{$M, B_0$}
    \State $R_1 \gets 2e_\epsilon \mid e_{B_0}$
    \State $\mathcal N \gets \{N \subseteq [n] : \rk(N) = r-1, |N| = r \}$
    \For{$N \in \mathcal{N}$}
    \Comment{$N$ is in increasing order}
    \State $C \gets \text{unique circuit in }N$
    \State $D \gets \text{unique cocircuit in }[n] \setminus N$
    \State $i\gets \text{first element of } C$
    \State $j\gets \text{first element of } D$
    \For{$m \in C \setminus \{ i \}$}
    \For{$n \in D \setminus \{ j \}$}
    \State $I \gets N \setminus \{ i,m \}$
    \Comment{$I$ inherits the order from $N$}
    \State $R_1 \gets R_1 \mid$ Cr$(I,i,m,j,n)$
    \EndFor
    \EndFor
    \EndFor
    \State \textbf{return} $R_1$
\EndProcedure
\end{algorithmic}
\label{alg:TutteGroup}
\end{algorithm}

Once we have the relations matrix $R_1$, we can compute its Smith normal form to obtain a set of minimal generators of $\T$, as well as a projection $\rho: G \twoheadrightarrow \T$.

We next recall the definition of the \emph{inner Tutte group}:

\begin{definition}[{\cite[7.3]{BAKER2021107883}}] \label{def:degreeMapOfTutteGroup}
With notation as in \Cref{def:TutteGroup}, there is a group homomorphism
\begin{align*}
    \widetilde{\deg} &: G \to \Z^n \\
    \epsilon \mapsto {\bf 0}, \qquad \prod X_B^{a_B} \mapsto & \sum a_B \bigg( \sum_{i\in B\backslash B_0} e_i -\sum_{i\in B_0\backslash B} e_i \bigg)
\end{align*}
It is straightforward to verify that $R(M, B_0)$ is contained in the kernel of this map.
Thus there is an induced map $\deg_{[n]} : \T_M \to \Z^n$, called the \emph{degree map} of the Tutte group.  
\end{definition}

\begin{definition}[{\cite[Definition 7.11]{BAKER2021107883}}] \label{def:innerTutteGroup}
The \emph{inner Tutte group} $\T_M^{(0)}$ is the kernel of $\deg_{[n]}$.
\end{definition}

For computational purposes, it is desirable to have a presentation of the inner Tutte group.
To achieve this, we find a section $s$ of the degree map, which splits the short exact sequence
\[
0 \longrightarrow \T_M^{(0)} \xrightarrow{\;\; i \;\;} \mathbb T_M \xrightarrow{\; \deg \;} \im(\deg) \longrightarrow 0
\]
We use the \emph{coordinating path algorithm} to construct the section $s$.
With $B_0 \in \mathcal B$ fixed as before, we construct a bipartite graph $\Gamma = \Gamma(M, B_0)$: the vertex set of $\Gamma$ is the bipartition $(B_0, [n] \setminus B_0)$ of $[n]$, and there is an edge between two vertices $a\in B_0$ and $b\notin B_0$ whenever $B_0 \setminus \{a\}\cup \{b\} \in \mathcal B$. 
Let $F = \{\{a_1, b_1\},\ldots, \{a_m, b_m\}\}$ be a spanning forest of $\Gamma$ (where $a_i \in B_0$, $b_i \not \in B_0$).
For $1 \le i \le m$, set $B_i := B_0\setminus \{a_i\} \cup \{b_i\}$ (which is a basis of $M$).
We exhibit the section $s$ by declaring $s(e_{b_i} - e_{a_i}) := X_{B_i}$ for all $1 \le i \le m$. 

Note that $\deg(X_{B_i}) = e_{b_i} - e_{a_i}$, which are $\Z$-linearly independent for $1 \le i \le m$, and since $\Z^n$ is free abelian, any subgroup is also free abelian.
Thus we will have that $s$ as above indeed defines a section of $\deg$ once we show that $\{e_{b_i} - e_{a_i}\}_{i=1}^{m}$ generates $\im(\deg)$:

\begin{lemma}
Let $M$ be a matroid, and $F = \{\{a_1, b_1\},\ldots, \{a_m, b_m\}\}$ a spanning forest of $\Gamma(M, B_0)$.
Then $\{e_{b_i} - e_{a_i}\}_{i=1}^{m}$ generates the image of $\deg_{[n]} : \T_M \to \Z^n$.
\end{lemma}

\begin{proof}
The Tutte group $\T_M$ is generated by $\epsilon$ and $X_B$, where $B$ runs over all bases of $M$.
It suffices to show that $\deg(X_B) \in \langle e_{b_i} - e_{a_i} : 1 \le i \le m \rangle \subseteq \Z^n$ for any $B \in \mathcal B$, which we do by induction on $|B\setminus B_0|$.

For the base case $|B\setminus B_0| = 1$: there exists an edge $e = \{a,b\}$ of $\Gamma$ with $a\in B_0$, $b\in B$ such that $B = B_0\setminus\{a\}\cup \{b\}$.
Since $F$ is a spanning forest of $\Gamma$, either $e\in F$, or there is a cycle in $e\cup F$.
If $e\in F$, then there is nothing to show.
Otherwise, let $C$ be a cycle in $e\cup F$, and let $\{i_1,...,i_s\}$ be the indices of edges of $C$ in order, as in the following picture:

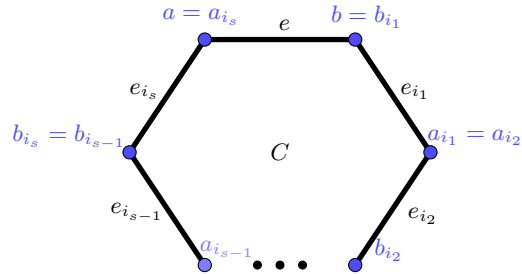
\begin{center}
\definecolor{xdxdff}{rgb}{0.49019607843137253,0.49019607843137253,1}
\definecolor{ududff}{rgb}{0.30196078431372547,0.30196078431372547,1}
\begin{tikzpicture}[line cap=round,line join=round,>=triangle 45,x=1cm,y=1cm]
\clip(-4,-0.5) rectangle (4,3.7);
\draw [line width=2pt] (-1,3)-- (1,3);
\draw [line width=2pt] (-1,3)-- (-2,1.5);
\draw [line width=2pt] (-2,1.5)-- (-1,0);
\draw [line width=2pt] (1,3)-- (2,1.5);
\draw [line width=2pt] (2,1.5)-- (1,0);
\begin{scriptsize}
\draw [fill=ududff] (-1,3) circle (2.5pt);
\draw[color=ududff] (-1.05,3.3) node {$a = a_{i_s}$};
\draw [fill=ududff] (1,3) circle (2.5pt);
\draw[color=ududff] (1.16,3.3) node {$b = b_{i_1}$};
\draw[color=black] (0.06,3.2) node {$e$};
\draw [fill=ududff] (-2,1.5) circle (2.5pt);
\draw[color=ududff] (-2.8,1.7) node {$b_{i_s} = b_{i_{s-1}}$};
\draw[color=black] (-1.8,2.31) node {$e_{i_s}$};
\draw [fill=xdxdff] (-1,0) circle (2.5pt);
\draw[color=xdxdff] (-0.7,0.2) node {$a_{i_{s-1}}$};
\draw[color=black] (-1.9,0.7) node {$e_{i_{s-1}}$};
\draw [fill=ududff] (2,1.5) circle (2.5pt);
\draw[color=ududff] (2.6,1.7) node {$a_{i_1} = a_{i_2}$};
\draw[color=black] (1.8,2.31) node {$e_{i_1}$};
\draw [fill=ududff] (1,0) circle (2.5pt);
\draw[color=ududff] (1.45,0.2) node {$b_{i_2}$};
\draw[color=black] (1.9,0.65) node {$e_{i_2}$};
\draw [fill=black] (-0.3,0) circle (1.5pt);
\draw [fill=black] (0,0) circle (1.5pt);
\draw [fill=black] (0.3,0) circle (1.5pt);
\draw[color=black] (0,1.5) node {$C$};
\end{scriptsize}
\end{tikzpicture}
\end{center}
Then note that $(e_b - e_a)+\sum_{j=1}^s (-1)^j(e_{b_{i_j}}-e_{a_{i_j}})=0$.

For the inductive step, let $B$ be a basis with $|B\setminus B_0| = k \ge 2$, and assume the claim holds for all bases that differ from $B_0$ by at most $k-1$ elements. 
By symmetric basis exchange \cite{brualdi1969comments}, there exist $a\in B_0\setminus B$ and $b\in B\setminus B_0$ such that $B' := B\setminus \{b\} \cup \{a\}$ and $B'' := B_0\setminus \{a\} \cup \{b\}$ are both bases. 
Then $\deg(X_B)-\deg(X_{B'}) = e_b-e_a = \deg(X_{B''})$,
$|B' \setminus B_0| = k-1$, and $|B'' \setminus B_0| = 1$, 
so by induction, the claim holds for $B$.
\end{proof}

With the section $s$, we can write the inner Tutte group as a quotient of the Tutte group, i.e. we have a projection $\T \twoheadrightarrow \T/\im(s) \cong \T^{(0)}$.
\Cref{alg:innerTutteGroup} constructs a lift of $\im(s)$ in $G$:

\begin{algorithm}
\caption{Additional relations for inner Tutte group}
\begin{algorithmic}[1]
\Require A matroid $M = ([n], \mathcal B)$ with a fixed basis $B_0 \in \mathcal B$
\Ensure A matrix $R_2$ over $\Z$ with $|\mathcal B|+1$ rows whose columns generate $\im(s)$ in $\T$
\Procedure{InnerTutteGroup}{$M, B_0$}
    \State $R_2 \gets \text{a $(|\mathcal{B}|+1) \times 0$ matrix over $\Z$}$
    \State $E \gets \{(a,b) \in B_0 \times ([n] \setminus B_0) \mid B_0 \setminus \{a\}\cup \{b\} \in \mathcal{B}\}$
    \State $\Gamma \gets \text{the bipartite graph with vertices $B_0 \sqcup ([n] \setminus B_0)$ and edges $E$}$
    \State $F \gets \text{a spanning forest of $\Gamma$}$
    \For{$\{a,b\}\in F$}
    \State $R_2 \gets R_2 \mid e_{B_0 \setminus \{ a \} \cup \{ b \}}$
    \Comment{columns of $R_2$ are standard basis vectors of $G$}
    \EndFor
    \State \textbf{return} $R_2$
\EndProcedure
\end{algorithmic}
\label{alg:innerTutteGroup}
\end{algorithm}

With the outputs $R_1$ and $R_2$ of \Cref{alg:TutteGroup,alg:innerTutteGroup} respectively, we can compute the Smith normal form of $R_1 \mid R_2$ to obtain a set of minimal generators of $\T^{(0)}$ and a projection $\rho^{(0)}: G \twoheadrightarrow \T^{(0)}$ factoring through $\rho$. 
This gives the multiplicative group and $\epsilon := \rho^{(0)}(e_\epsilon)$ for $F_M$, so it remains to compute the additive relations.
By \cite[Theorem 4.18]{Baker-Lorscheid20}, the 3-term additive relations are given by modular quadruples of hyperplanes, i.e. sets of 4 hyperplanes of $M$ that intersect in a corank 2 flat.
\Cref{alg:additiveRelations} constructs these additive relations:

\begin{algorithm}
\caption{Additive relations for the foundation}
\begin{algorithmic}[1]
\Require A matroid $M$ of rank $r$ with set of flats $\mathcal F$, and the projection $\rho^{(0)} : G \twoheadrightarrow \T^{(0)}$
\Ensure The list of fundamental pairs of $F_M$
\Procedure{FundamentalPairsFoundation}{$M, \mathcal{F}, \rho$}
    \State FP $\gets \{ \}$
    \State $\mathcal H \gets \{ F \in \mathcal F : \rk F = r - 1 \}$
    \State $\mathcal{F}_2 \gets \{ F \in \mathcal F : \rk F = r - 2 \}$
    \For{$X \in \mathcal{F}_2$}
    \State $I \gets \text{a maximal independent set in $X$}$
    \State $S \gets \{ H \in \mathcal{H} : X \subseteq H \}$
    \For{all $\{ H_1, H_2, H_3, H_4 \} \in \binom{S}{4}$}
    \State $a_i \gets$ any element in $H_i \setminus X$, for $i = 1, \ldots, 4$
    \State FP $\gets$ FP $\cup \; \{\{\rho^{(0)}(\Cr(I,a_1,a_2,a_3,a_4)), \rho^{(0)}(\Cr(I,a_1,a_3,a_2,a_4))\}\}$
    \EndFor
    \EndFor
    \State \textbf{return} FP 
\EndProcedure
\end{algorithmic}
\label{alg:additiveRelations}
\end{algorithm}

We note that the cross-ratios $\Cr(I,a_1,a_2,a_3,a_4)$ in \Cref{alg:additiveRelations} are \emph{non-degenerate} in the sense that any 2 of $\{ a_1, \ldots, a_4 \}$, along with $I$, forms a basis of $M$.
This is not the case for the cross-ratios $\Cr(I,i,m,j,n)$ in \Cref{alg:TutteGroup}, which are \emph{degenerate} -- see \cite[7.1]{BAKER2021107883}.

\subsection{Examples of pastures/foundations} \label{ssec:pastureExamples}

\begin{definition} [{\cite[Theorem 5.8]{Baker-Lorscheid20}}] \label{def:hexTypes}
There are 4 specific pastures, each with one hexagon, called $\U$, $\D$, $\H$, and $\F_3$, which are foundations of particular matroids:
\begin{itemize}
    \item $\U = \F_1^\pm \langle x,y\rangle \sslash \{x+y-1\}$ is the foundation of the uniform matroid $U(2, 4)$, called the \emph{near-regular partial field}.
    \item $\D = \F_1^\pm \langle x\rangle \sslash \{x+x-1\}$ is the foundation of the non-Fano matroid $F_7^-$, called the \emph{dyadic partial field}.
    \item $\H = \F_1^\pm \langle x\rangle \sslash \{x+x^{-1}-1, x^3 = -1\}$ is the foundation of the affine geometry $AG(2,3)$, called the \emph{hexagonal partial field}.
    \item $\F_3 = \F_1^\pm \sslash \{-1-1-1\}$ is the foundation of the matroid $T_8$ (\cite[Exercise 2.2.9]{Oxley92}), and is the finite field with 3 elements.
\end{itemize}
For any pasture $P$ with a fundamental pair $(a, b) \in \FP(P)$, there is a pasture morphism $\phi : \U \to P$ sending $x \mapsto a$, $y \mapsto b$.
We say that the hexagon corresponding to $(a, b)$ is of:
\begin{itemize}
    \item \emph{type} $\F_3$, if $\phi$ factors through $\U \to \F_3$;
    \item \emph{type} $\D$, if $\phi$ factors through $\U \to \D$, but not through $\U \to \F_3$;
    \item \emph{type} $\H$, if $\phi$ factors through $\U \to \H$, but not through $\U \to \F_3$;
    \item \emph{type} $\U$, if $\phi$ does not factor through $\U \to \D$ nor $\U \to \H$.
\end{itemize}
\end{definition}

Typically, the type of a hexagon is reflected in the number of distinct (fundamental) elements it contains (alternatively, the sizes of its $S_3$-orbits).
For instance, a hexagon is: of type $\F_3$ iff it consists of 1 distinct element; of type $\D$ iff it has a unique fundamental pair whose elements coincide; and of type $\H$ iff it consists of 2 distinct elements and is not of type $\D$ (or equivalently, consists of 1 distinct fundamental pair and is not of type $\F_3$). Thus e.g. a hexagon with $\ge 4$ distinct elements must be of type $\U$.
See \cite[Section 3.3]{Baker-Lorscheid18} for more.

Incidentally, we describe how to view a finite field as a pasture (which will be useful for applications of the algorithms in \Cref{sec:morphismAlgorithm}, see \Cref{sec:matrixReps}):

\begin{ex}[Finite fields] \label{ex:finiteFields}
For any finite field $\F_q$ with $q$ elements, the multiplicative group $\F_q^\times$ is cyclic of order $q-1$.
We have $\epsilon = -1$ (= 1 if char $\F_q = 2$).
For any $\beta \in \F_q \setminus \{0, 1\}$, there is a unique hexagon $( ( \beta, 1-\beta ), ( \tfrac{1}{\beta}, \tfrac{\beta-1}{\beta} ), ( \tfrac{1}{1-\beta}, \tfrac{\beta}{1-\beta} ) )$ containing $\beta$, and all hexagons of $\F_q$ arise in this way.

It is a straightforward exercise (left to the reader) to check that the finite field $\F_q$ has:
\begin{itemize}
    \item 1 hexagon of type $\F_3 \iff 3 \mid q$, consisting of $1$ element of $\F_q$,
    \item 1 hexagon of type $\H \iff q \equiv 1 \pmod 3$, consisting of $2$ elements of $\F_q$,
    \item 1 hexagon of type $\D \iff 2 \nmid q$ and $3 \nmid q$, consisting of $3$ elements of $\F_q$,
    \item All remaining elements of $\F_q \setminus \{0, 1\}$ are in hexagons of type $\U$.
\end{itemize}
\end{ex}

\section{Morphism Algorithm} \label{sec:morphismAlgorithm}

\subsection{Overview} \label{ssec:overview}
We now turn our attention to the following problem: given two finitely presented pastures $P_1, P_2$, provide an algorithm to compute all pasture morphisms $P_1 \to P_2$.
Our approach to solving this problem constitutes the main novel contribution of this paper.
By \Cref{lem:Hom}, a pasture morphism $P_1 \to P_2$ is equivalent to an abelian group homomorphism $P_1^\times \to P_2^\times$ that sends fundamental pairs in $P_1^\times$ to fundamental pairs in $P_2^\times$, i.e. we identify
\[
\Mor(P_1, P_2) \leftrightarrow \{ f \in \Hom(P_1^\times, P_2^\times) : (f(x), f(y)) \in \FP(P_2) \; \forall (x, y) \in \FP(P_1) \}
\]
(henceforth we denote pasture morphisms by $\Mor$, and reserve $\Hom$ for (abelian) group homomorphisms).
We may thus view the problem of computing pasture morphisms as a restricted homomorphism problem, i.e. finding all group homomorphisms where the images of certain elements are constrained to lie in a specified set.

\begin{rem}
Throughout, we always work with finitely generated abelian groups represented in Smith normal form, i.e. every finitely generated abelian group is presented as the cokernel of a diagonal matrix over $\Z$ (note that in such a presentation, the group is finite if and only if all diagonal entries are nonzero).
We also assume that subroutines for the following problems are given:

\begin{itemize}
    \item (HOM) Given finite abelian groups $G_1$ and $G_2$, compute the finite set $\Hom(G_1, G_2)$
    \item (MODSOLVE) Given $n \in \Z$ and matrices $A \in \Z^{r \times s}$, $B \in \Z^{1 \times s}$, determine if a solution $X \in \Z^{1 \times r}$ to $XA = B$ over $\Z/n\Z$ exists, and if so, return one.
 \end{itemize}
\end{rem}

We begin by reducing to the case that $P_1^\times$ is generated by its fundamental elements.
Consider the subpasture $P_0 \subseteq P_1$, defined by $P_0^\times := \langle \FE(P_1) \rangle$, $\epsilon_{P_0} := \epsilon_{P_1}$ (which is in $P_0^\times$ if $\FE(P_1) \ne \emptyset$) and $\FP(P_0) := \FP(P_1)$.
There is a short exact sequence of abelian groups
\[
  0 \to P_0^\times \xrightarrow{i} P_1^\times \xrightarrow{p} P_1^\times/P_0^\times \to 0.
\]
Applying $\Hom(\cdot, P_2^\times)$ gives a long exact sequence of abelian groups
\[
  0 \to \Hom(P_1^\times/P_0^\times, P_2^\times) \xrightarrow{p^*} \Hom(P_1^\times, P_2^\times) \xrightarrow{i^*} \Hom(P_0^\times, P_2^\times) \xrightarrow{\delta} \Ext^1(P_1^\times/P_0^\times,P_2^\times) \to \cdots
\]
It follows that $\Mor(P_1, P_2)$ is a union of cosets of $p^*(\Hom(P_1^\times/P_0^\times, P_2^\times))$, with representatives consisting of (preimages of) elements in $\Mor(P_0, P_2)$ which are in the image of $i^*$.
As $\im(i^*) = \ker(\delta)$ by exactness, this gives a computable description of $\Mor(P_1, P_2)$, once $\Mor(P_0, P_2)$ is known.

Thus (replacing $P_1$ by $P_0$) we assume henceforth that $P_1^\times = \langle \FE(P_1) \rangle$ (note in any case that by {\cite[p. 9]{Baker-Lorscheid20}} this is automatic if $P_1$ is the foundation of a matroid, which in view of \Cref{thm:UP} is the case of most interest).
Although this reduces the problem to a finite computation in principle (as there are now only finitely many possibilities for the image of each generator), there is still nontrivial work involved in giving an efficient algorithm in practice, as the next remark shows.

\begin{rem} \label{rem:naiveComp}
Let $P$ be the Pappus matroid.
The foundation of $P$ is a pasture $F_P$ with $11$ hexagons, all of type $\U$.
To compute representations of $P$ over $\F_8$ (which has $1$ hexagon, of type $\U$), one could in theory search over all $6^{11} = 362797056$ candidate set maps sending (a fundamental pair of) each hexagon in $F_P$ to a fundamental pair in $\F_8$.
However, it is not clear a priori which of these $6^{11}$ candidates come from group homomorphisms $(F_P)^\times \to (\F_8)^\times$, and in any case such a search would be extremely inefficient.
\end{rem}

We thus turn towards exploiting additional algebraic structure in the problem.
By the structure theorem of finitely generated abelian groups, for $i = 1, 2$ we can decompose $P_i^\times \cong T_i \oplus F_i$ where $F_i$ is free and $T_i$ is torsion (finite).
We denote the projection maps onto the free resp. torsion parts by $\pi_{F_i}$ resp. $\pi_{T_i}$.
This gives a direct sum decomposition 
\begin{equation} \label{eq:decomp}
    \Hom(P_1^\times, P_2^\times) \cong \Hom(T_1, P_2^\times) \oplus \Hom(F_1, P_2^\times) = \Hom(T_1, T_2) \oplus \Hom(F_1, P_2^\times)
\end{equation}
(note that $\Hom(T_1, F_2) = 0$).
Let $r_i$ denote the free rank of $F_i$, and consider the Smith normal form of $P_i^\times$, i.e. a presentation of $P_i^\times$ as the cokernel of an $(n_i + r_i) \times (n_i + r_i)$ diagonal matrix over $\Z$, where all $n_i$ nonzero diagonal entries appear first, ordered by consecutive divisibility.
In this way, the decomposition (\ref{eq:decomp}) realizes elements in $\Hom(P_1^\times, P_2^\times)$ as block matrices
\begingroup
\renewcommand*{\arraystretch}{1.8}
\begin{equation} \label{eq:matBlockForm}
\left[ \begin{array}{@{}c|c@{}}
  \begin{matrix}
  \psi_{n_2 \times n_1} \\
  \hline
  0_{r_2 \times n_1}
  \end{matrix}
   &
  C_{(n_2+r_2) \times r_1}
\end{array}\right]
\end{equation}
\endgroup
where $\psi \in \Hom(T_1, T_2)$, $C \in \Hom(F_1, P_2^\times)$.
We will obtain all elements in $\Mor(P_1, P_2)$ by determining all possibilities for $\psi$ and $C$.

We first address the simpler step of obtaining $\psi$.
Note that $\epsilon_i \in P_i^\times$ is always torsion, as $\epsilon_i^2 = 1$ in $P_i$.
Since $T_1, T_2$ are finite, we may use (HOM) to compute $\Hom(T_1, T_2)$, keeping only the homomorphisms which send $\epsilon_1$ to $\epsilon_2$.
In practice, this is not an issue for efficiency, as foundations of most common matroids have multiplicative torsion groups of small order (e.g. $\le 6$).
In the following, we will assume that a choice of $\psi$ has been fixed.

The harder step is computing the relevant subset of $\Hom(F_1, P_2^\times)$ for pasture morphisms, and the rest of the section will be devoted to this.
We now describe a general outline of the strategy.
For simplicity of notation, set $r := r_1 = \rk(F_1)$.
Notice that although $\Hom(F_1, P_2^\times) \cong (P_2^\times)^r$ is trivial to compute as a group, one is only interested in a particular (finite) subset.
Moreover, one does not have good control over the potential images (in $P_2^\times$) of minimal generators of $F_1$.

Our approach therefore is to first compute a \emph{full rank sublattice} of $F_1$, generated by projections of fundamental elements of $P_1$, see \Cref{ssec:FRS}.
This consists of $r$ fundamental elements $\{ x_1, \ldots, x_r \}$ in $P_1$, such that $\{ \pi_{F_1}(x_1), \ldots, \pi_{F_1}(x_r) \}$ generate a (free) subgroup $H \subseteq F_1$ of rank $r$.
We use this sublattice to traverse the search space of candidate morphisms using a depth-first search, see \Cref{ssec:DFS}.
Finally, we assemble the ingredients (i.e. build the possibilities for $C$ in (\ref{eq:matBlockForm})) to obtain all possible morphisms in \Cref{ssec:assembly}.

\subsection{Full Rank Sublattice} \label{ssec:FRS}

In this section we give a construction for a full rank sublattice of the free part of the multiplicative group of a pasture $P$ (assuming $P$ is finitely related and $P^\times = \langle \FE(P) \rangle$), which is generated by (projections of) fundamental elements.
In general there are many possible choices for such a sublattice, as well as its generators.
Although any such sublattice reveals insight into the structure of the pasture, and can be used to compute morphisms, the algorithm we give here has definite advantages over a ``random'' full rank sublattice, cf. \Cref{ssec:DFS,ssec:efficiency}.

Naively, one might try to iteratively build up a full rank sublattice by a ``greedy'' algorithm: i.e. choose fundamental pairs such that adding the pair to the previously generated subgroup increases the rank by $2$.
However, this approach turns out to be precisely the opposite of what is desirable, for computing morphisms.
Indeed, rather than searching for fundamental pairs which are independent, we search for fundamental pairs which have nontrivial dependency relations with the previous pairs.
To be precise, we introduce the following terminology:

\begin{definition} \label{def:types}
Let $L \subseteq \Z^r$ be any subgroup, with $\rk(L) =: n$, and $u, v \in \Z^r$. We say that the pair $( u, v )$ (with respect to $L$) is:
\begin{itemize}
    \item \emph{Type 1}: if $\rk(L + \langle u \rangle) = n$, $\rk(L + \langle v \rangle) = n+1$,
    \item \emph{Type 2}: if $\rk(L + \langle u \rangle) = \rk(L + \langle v \rangle) = \rk(L + \langle u, v \rangle) = n+1$,
    \item \emph{Type 3}: if $\rk(L + \langle u \rangle) = \rk(L + \langle v \rangle) = n+1$, $\rk(L + \langle u, v \rangle) = n+2$,
    \item \emph{Type 4}: if $\rk(L + \langle u \rangle) = \rk(L + \langle v \rangle) = n$.
\end{itemize}
\end{definition}

Note that any pair of vectors is one of the 4 types listed above (up to reordering in the case of type 1).
In essence, our algorithm for producing a full rank sublattice proceeds by choosing fundamental pairs in the order of preference listed by the types above, as detailed in \Cref{alg:fullRankSublattice}.
Although type 4 pairs do not increase the rank (and thus are never chosen to be in the full rank sublattice $L$), we still record the presence of any type 4 pairs whenever they occur at an intermediate step, to be used in \Cref{alg:depthFirstSearch}.

\begin{algorithm}
\caption{Full rank sublattice of the free part of a pasture}
\label{alg:fullRankSublattice}
\begin{algorithmic}[1]
\Require A pasture $P$ of free rank $r = \rk(P^\times)$, with $P^\times = \langle \FE(P) \rangle$
\Ensure A list $\mathcal{L} = \{x_1, \ldots, x_r \} \subseteq \FE(P)$ with $\{ \pi_F(x_1), \ldots, \pi_F(x_r) \}$ linearly independent, a list $\mathcal{R}$ of rule pairs, a list $\mathcal{S}$ of lists of type 4 pairs (all lists are ordered)
\Procedure{FullRankSublattice}{P}
    \State $r \gets \rk(P^\times)$
    \State $\mathcal{L} \gets \{ \}$
    \State $\mathcal{R} \gets \{ \}$
    \State $\mathcal{S} \gets \{ \}$
    \State $n \gets 0$
    \Repeat
        \If{$\exists (u,v) \in \FP(P)$ with $(\pi_F(u), \pi_F(v))$ type 1 wrt $\langle \pi_F(\mathcal{L}) \rangle$}
        \State $\mathcal{L} \gets \mathcal{L} \cup \{v\}$
            \State $\exists (c_1, \ldots c_{n+1}) \in \Z^{n+1} \setminus \{ \bfz \}$: $c_{n+1} \pi_F(u) + \sum_{i=1}^n c_i \pi_F(x_i) = 0$
            \State $\mathcal{R} \gets \mathcal{R} \cup \{ \{c_{n+1} u + \sum_{i=1}^n c_i x_i, (c_1, \ldots c_{n+1})\} \}$
        \ElsIf{$\exists (u,v) \in \FP(P)$ with $(\pi_F(u), \pi_F(v))$ type 2 wrt $\langle \pi_F(\mathcal{L}) \rangle$}
            \State $\mathcal{L} \gets \mathcal{L} \cup \{v\}$
            \State $\exists (c_1, \ldots c_{n+2}) \in \Z^{n+2} \setminus \{ \bfz \}$: $c_{n+1} \pi_F(u) + c_{n+2} \pi_F(v) + \sum_{i=1}^n c_i \pi_F(x_i) = 0$
            \State $\mathcal{R} \gets \mathcal{R} \cup \{ \{c_{n+1} u + c_{n+2} v + \sum_{i=1}^n c_i x_i, (c_1, \ldots c_{n+2})\} \}$
        \ElsIf{$\exists (u,v) \in \FP(P)$ with $(\pi_F(u), \pi_F(v))$ type 3 wrt $\langle \pi_F(\mathcal{L}) \rangle$}
            \State $\mathcal{L} \gets \mathcal{L} \cup \{u, v\}$
            \State $\mathcal{R} \gets \mathcal{R} \cup \{ \{ 0, 0 \}, \{ 0, 1 \}\}$
            \State $\mathcal{S} \gets \mathcal{S} \cup \{ \{ \} \}$
        \EndIf
        \State $n \gets \rk(\langle \pi_F(\mathcal{L}) \rangle)$
        \State $\mathcal{S} \gets \mathcal{S} \cup \{ \{ ( u, v ) \in \FP(P) \mid ( \pi_F(u), \pi_F(v) ) \text{ type } 4 \text{ wrt } \langle \pi_F(\mathcal{L}) \rangle \} \}$

    \Until{$n = r$}
    \State \textbf{return $\{\mathcal{L}, \mathcal{R}, S\}$} 
\EndProcedure
\end{algorithmic}
\end{algorithm}

\begin{rem} \label{rem:alg1Remarks}
A few remarks concerning \Cref{alg:fullRankSublattice} are in order:

i) Termination is guaranteed by the assumption $P^\times = \langle \FE(P) \rangle$.

ii) When choosing the coefficients $c_i$ (lines 10 and 14), we always choose a primitive set of coefficients, i.e. $(c_1, c_2, \ldots) = 1$.

iii) At the end of \Cref{alg:fullRankSublattice}, every fundamental pair of $P$ will appear either as part of $\mathcal{L}$, or as a type 4 pair in $\mathcal{S}$: at the last iteration of the loop (when $n = r$), every remaining fundamental pair will become of type $4$.

iv) Once computed, the data of the full rank sublattice $\{\mathcal{L}, \mathcal{R}, S\}$ can be cached as part of the source pasture $P_1$, and any subsequent computation of morphisms $\Mor(P_1, P_2)$ (for various pastures $P_2$) may reuse the data $\{\mathcal{L}, \mathcal{R}, S\}$.
\end{rem}

\subsection{Depth-First Search} \label{ssec:DFS}

In this section, we describe an algorithm for generating candidate maps in $\Hom(H, P_2^\times)$, where $H := \langle \pi_{F_1}(x_1), \ldots, \pi_{F_1}(x_r) \rangle$ (with $x_i \in \FE(P_1)$) is the full rank sublattice in \Cref{ssec:FRS}.
Recall that a choice of $\psi \in \Hom(T_1, T_2)$ has been fixed.

As mentioned in \Cref{rem:naiveComp}, the search space for candidate maps can be extremely large, even for matroids of moderate size (e.g. $\le 10$ elements).
Thus, the primary goal is to cut down the search space as much as possible.
We will see (\Cref{ssec:efficiency}) how the choice of sublattice in \Cref{ssec:FRS} is motivated by this goal.
Intuitively, by choosing type 1 (and type 2) pairs over type 3 pairs, as well as remembering the type 4 pairs at each step, one can impose more conditions upon a candidate morphism, thereby eliminating portions of the search space.

To be precise, we iteratively build a (rooted) tree representing the search space of candidate maps, see \Cref{fig:tree}.

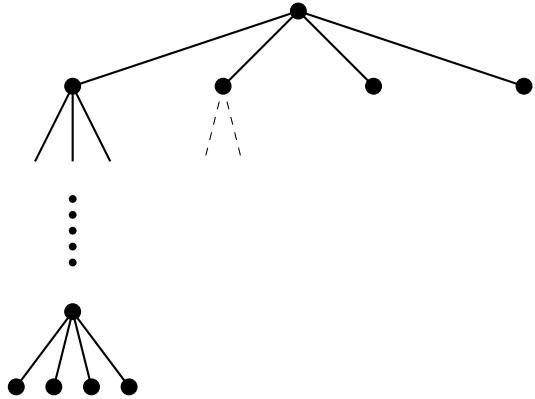
\begin{figure}[h]
    \centering
    \begin{tikzpicture}
\draw[thick] (0,5) -- (-3,4);
\draw[thick] (0,5) -- (-1,4);
\draw[thick] (0,5) -- (1,4);
\draw[thick] (0,5) -- (3,4);
\draw[thick] (-3,4) -- (-3.5,3);
\draw[thick] (-3,4) -- (-3,3);
\draw[thick] (-3,4) -- (-2.5,3);
\draw[dashed] (-1,4) -- (-1.25,3);
\draw[dashed] (-1,4) -- (-0.75,3);
\draw[thick] (-3,1) -- (-3.75,0);
\draw[thick] (-3,1) -- (-3.25,0);
\draw[thick] (-3,1) -- (-2.75,0);
\draw[thick] (-3,1) -- (-2.25,0);
\draw[line width=3pt, line cap=round, dash pattern=on 0pt off 2\pgflinewidth] (-3,2.5) -- (-3,1.5);
\filldraw (0,5) circle (3pt);
\filldraw (-3,4) circle (3pt);
\filldraw (-1,4) circle (3pt);
\filldraw (1,4) circle (3pt);
\filldraw (3,4) circle (3pt);
\filldraw (-3,1) circle (3pt);
\filldraw (-3.75,0) circle (3pt);
\filldraw (-3.25,0) circle (3pt);
\filldraw (-2.75,0) circle (3pt);
\filldraw (-2.25,0) circle (3pt);
\end{tikzpicture}
\caption{Search tree}
\label{fig:tree}
\end{figure}

Recalling that $r = \rk(H) = \rk(P_1^\times)$, the tree will have $r+1$ levels, with the root at level $0$.
The nodes at the bottom level $r$ correspond to (fully specified) candidate maps, to be handled in \Cref{ssec:assembly}.
In general, the nodes at each level correspond to partial maps, where the images of an initial segment of generators of $H$ have been specified.

We traverse the tree using a depth-first search.
In our setting, this is preferable over breadth-first search for a number of reasons, the most important of which is that it allows us to use the information obtained in \Cref{alg:fullRankSublattice} to eliminate branches in the tree from consideration.
Additionally, it allows for an algorithm which is far more efficient in terms of memory (linear vs exponential in $r$).

Initially, the tree consists of a single node (the root), which is taken as the current node.
After deletion of a node, the current node passes to the next node in the same level (if such a node exists).
The search proceeds via cases, depending on the level of the current node.
\begin{itemize}
    \item If the current level is $< r$ and is nonempty, then populate the next level, and set the current node to be the first node in the new level.
    \item If the current node is at level $r$, then check to see if it gives any pasture morphisms (see \Cref{ssec:assembly}), and then delete it.
    \item If the current level is empty, then backtrack up one level and delete the first node.
    \item The search terminates when the root node is deleted.
\end{itemize}

\begin{rem} \label{rem:rankZeroSearch}
Note that if $r = 0$ (e.g. as in the case of $P_1 = \H$ or $\F_3$), then $P_1^\times = T_1$ is torsion, hence the map $P_1^\times \to P_2^\times$ is already specified (as $\psi$). Thus we may assume that $r > 0$, so that the first case above indeed occurs.
\end{rem}

It remains to explain how new levels of the tree are populated.
Suppose that the current level is $n$ ($0 \le n < r$), and we want to populate the tree in level $n+1$ with a list $\mathcal{T}_{n+1}$ of candidates.
Inductively we may assume that levels $\mathcal{T}_1, \ldots, \mathcal{T}_{n}$ have already been populated, and set $y_1, \ldots, y_{n}$ as the first elements of $\mathcal{T}_1, \ldots, \mathcal{T}_{n}$ respectively.
We construct $\mathcal{T}_{n+1}$ as the list of all possible images of $x_{n+1}$, assuming that the images of $x_1, \ldots, x_{n}$ are $y_1, \ldots, y_{n}$ respectively (i.e. $x_1$ is mapped to $y_1$, $x_2$ is mapped to $y_2$, etc.).
Recall from \Cref{alg:fullRankSublattice} that corresponding to $x_{n+1}$, there is a \emph{rule pair} $r_{n+1} \in \mathcal{R}$, as well as a list of type 4 pairs $s_{n+1} \in \mathcal{S}$.
We have the following cases:

\begin{itemize}
    \item If $r_{n+1} = \{\tau, (c_1, \ldots, c_{n+1})\}$, then $x_{n+1}$ is the second member of a type 1 pair; set 
    $\mathcal{T}_{n+1} = \{ y \in \FE(P_2) \mid \exists x \text{ a partner of } y \text{ s.t. } c_1 y_1 + \ldots + c_{n} y_{n} + c_{n+1} x = \psi(\tau) \}$.
    \item If $r_{n+1} = \{\tau, (c_1, \ldots, c_{n+2})\}$, then $x_{n+1}$ is the first member of a type 2 pair; set
    $\mathcal{T}_{n+1} = \{ y \in \FE(P_2) \mid \exists x \text{ a partner of } y \text{ s.t. } c_1 y_1 + \ldots + c_{n} y_{n} + c_{n+1} x + c_{n+2} y = \psi(\tau) \}$.
    \item If $r_{n+1} = \{ 0, 0 \}$, then $x_{n+1}$ is the first member of a type 3 pair; set $\mathcal{T}_{n+1} = \FE(P_2)$.
    \item If $r_{n+1} = \{ 0, 1 \}$, then $x_{n+1}$ is the second member of a type 3 pair; set $\mathcal{T}_{n+1} = \{ y \in \FE(P_2) \mid y \text{ is a partner of } y_n \}$.
\end{itemize}

Next, we remove from $\mathcal{T}_{n+1}$ any candidates which do not send all type 4 pairs in $s_{n+1}$ to fundamental pairs in $P_2$.
This is possible using the same procedure for type 1 and type 2 pairs, since both elements in a type 4 pair are linearly dependent (over $\Z$) on $x_1, \ldots, x_n$.
Once this is done, we adjoin $\mathcal{T}_{n+1}$ as the $(n+1)^\text{st}$ level of the tree, and continue the search.

To see why this works, consider the case of a type 1 pair $( u, v )$.
When $v$ is added in \Cref{alg:fullRankSublattice}, there is a dependence relation $c_1 \pi_{F_1}(x_1) + \ldots + c_{n} \pi_{F_1}(x_{n}) + c_{n+1} \pi_{F_1}(u) = 0$.
This means that $\pi_{F_1}(c_1 x_1 + \ldots + c_{n} x_{n} + c_{n+1} u) = 0$, i.e. $\tau := c_1 x_1 + \ldots + c_{n} x_{n} + c_{n+1} u$ is torsion in $P_1^\times$.
Now any pasture morphism $f \in \Mor(P_1, P_2)$ must send $u$ to a fundamental element $x \in \FE(P_2)$, hence
\begin{align*}
    \tau &= c_1 x_1 + \ldots + c_n x_n + c_{n+1} u \\
    \implies f(\tau) &= f(c_1 x_1 + \ldots + c_n x_n) + f(c_{n+1} u) \\
    \implies \psi(\tau) &= c_1 f(x_1) + \ldots c_n f(x_n) + c_{n+1} f(u) \\
    \implies \psi(\tau) &= c_1 y_1 + \ldots + c_n y_n + c_{n+1} x
\end{align*}
After obtaining all fundamental elements $x \in \FE(P_2)$ satisfying this condition, we take all partners of such $x$ as all possible images of $v$.

We give the full procedure in \Cref{alg:depthFirstSearch}.

\algrenewcommand\algorithmicindent{0.9em}%

\begin{algorithm}
\caption{Depth-first search for pasture morphisms}
\label{alg:depthFirstSearch}
\begin{algorithmic}[1]
\Require Pastures $P_1, P_2$, full rank sublattice data $\{\mathcal{L}, \mathcal{R}, S\}$ for $P_1$, $\psi \in \Hom(T_1, T_2)$
\Ensure A list of pasture morphisms $\mathcal{M} \subseteq \Hom(P_1^\times, P_2^\times)$
\Procedure{DepthFirstSearch}{$P_1, P_2, \mathcal{L}, \mathcal{R}, \mathcal{S}, \psi$}
\State $\mathcal{M} \gets \{ \}$
\State $\mathcal{T}_0 \gets \{0\}$
\State $l \gets 0$
\State $x_i \gets i^\text{th}$ element of $\mathcal{L}$, for $i = 1, \ldots, r$
 \While{$\mathcal{T}_0 \neq \{ \}$}
    \If{$\mathcal{T}_l = \{ \}$}
    \State $l \gets l-1$
    \State $\mathcal{T}_l \gets \mathcal{T}_l \setminus \{ y_l \}$ 
    \ElsIf{$l < r$}
    \State $y_l \gets$ first element of $\mathcal{T}_l$
    \State $l \gets l+1$
    \If{$\mathcal{R}_l = \{\tau_l,(c_1, \ldots c_{l})\}$}
    \Comment{type 1}
    \State $\mathcal{P}_l \gets \{ x \in \FE(P_2) : \psi(\tau_l) = \sum_{i=1}^{l-1} c_i y_{i} + c_{l} x \}$
    \State $\mathcal{T}_l \gets \{ y\in \FE(P_2) : \exists x \in \mathcal{P}_l$ s.t. $( x,y ) \in \FP(P_2)\}$ 
    \ElsIf{$\mathcal{R}_l = \{\tau_l,(c_1, \ldots c_{l+1})\}$}
    \Comment{type 2}
    \State $\mathcal{P}_l \gets \{ ( x,y ) \in \FP(P_2) : \psi(\tau_l) = \sum_{i=1}^{l-1} c_i y_{i} + c_{l} x + c_{l+1}y \}$
    \State $\mathcal{T}_l \gets \{ x \in \FE(P_2) : \exists y$ s.t. $( x,y ) \in \mathcal{P}_l\}$
    \ElsIf{$\mathcal{R}_l = \{ 0, 0 \}$}
    \Comment{first member of type 3}
    \State $\mathcal{T}_l \gets \FE(P_2)$
    \ElsIf{$\mathcal{R}_l = \{ 0, 1 \}$}
    \Comment{second member of type 3}
    \State $\mathcal{T}_l \gets$ Partners of $y_{l-1}$
    \EndIf
    \For{$y \in \mathcal{T}_l$}
    \For{$( u,v ) \in \mathcal{S}_l$} 
    \Comment type 4
    \State $\exists (c_0,\ldots, c_{l})$: $c_0\pi_{F_1}(u) = \Sigma_{i=1}^{l} c_i\pi_{F_1}(x_i); y_u \gets \psi(c_0u - \Sigma_{i=1}^{l} c_ix_i) +\Sigma_{i=1}^{l-1} c_iy_{i}+c_{l}y$
    \State $\exists (d_0,\ldots, d_{l})$: $\! d_0\pi_{F_1}(v) = \Sigma_{i=1}^{l} d_i\pi_{F_1}(x_i); y_v \gets\!\psi(d_0v - \Sigma_{i=1}^{l} d_ix_i) + \Sigma_{i=1}^{l-1} d_iy_{i}+d_{l}y$
    \If{$( c_0 z_1, d_0 z_2 ) \ne ( y_u, y_v ) \; \forall ( z_1, z_2 ) \in \FP(P_2)$}
    \State $\mathcal{T}_l \gets \mathcal{T}_l \setminus \{ y \}$; {\bf break}
    \EndIf
    \EndFor
    \EndFor
    \ElsIf{$l=r$} 
    \For{$y \in \mathcal{T}_r$}
    \State $\mathcal{M} \gets \mathcal{M} \cup \{ \text{assembled morphisms from} (y_{1} \mid \ldots \mid y_{r-1} \mid y) \text{ as in \Cref{ssec:assembly}} \}$
    \EndFor
    \State $\mathcal{T}_r \gets \{ \}$
    \EndIf
    \EndWhile
    \State \textbf{return} $\mathcal{M}$
\EndProcedure
\end{algorithmic}
\end{algorithm}

\subsection{Final assembly} \label{ssec:assembly}

At the bottom level of the tree, we have fixed a map in $\Hom(H, P_2^\times)$, sending the basis of $H$ consisting of fundamental elements in $P_1$ to fundamental elements in $P_2$.
We now describe how to obtain pasture morphisms in $\Mor(P_1, P_2)$ from this (recall that as in \Cref{ssec:DFS}, a choice of $\psi \in \Hom(T_1, T_2)$ has also been fixed).
We continue to use the notation of the previous subsections (cf. (\ref{eq:decomp}), (\ref{eq:matBlockForm})).
For the remainder of this subsection, let $A$
denote the $(n_1 + r_1) \times r_1$ matrix whose columns are the fundamental elements $\{x_1, \ldots, x_{r_1}\}$ of $P_1$ generating the full rank sublattice $H$ (as chosen by \Cref{alg:fullRankSublattice}).
We view $A$ as a block matrix
\[
A = \begin{bmatrix}
x_1 & \ldots & x_{r_1}
\end{bmatrix} = \begin{bmatrix}
\pi_{T_1}(A) \\
\pi_{F_1}(A)
\end{bmatrix}
\]
where $\pi_{T_1}(A)$ is $n_1 \times r_1$, and $\pi_{F_1}(A)$ is $r_1 \times r_1$ (note that $\pi_{F_1}(A)$ concretely realizes the injection $H \hookrightarrow F_1$).
We also set 
\[
B := \begin{bmatrix}
y_1 & \ldots & y_{r_1}
\end{bmatrix} = 
\begin{bmatrix}
\pi_{T_2}(B) \\
\pi_{F_2}(B)
\end{bmatrix}
\]
as the matrix representing a given $f\in \Hom(H, P_2^\times)$ (as chosen by \Cref{alg:depthFirstSearch}), where $\pi_{T_2}(B)$ is $n_2 \times r_1$, and $\pi_{F_2}(B)$ is $r_2 \times r_1$ (note: the columns $y_1, \ldots, y_{r_1}$ of $B$ are fundamental elements of $P_2$, which are candidate images of $x_1, \ldots, x_{r_1}$).
Recall that the goal is to find all possibilities for $C$ in (\ref{eq:matBlockForm}).
We also view $C$ as a block matrix, with the same size blocks as $B$:
\[
C = 
\begin{bmatrix}
\pi_{T_2}(C) \\
\pi_{F_2}(C)
\end{bmatrix}.
\]

\begin{prop} \label{prop:assembly}
Let
$M := \begin{bmatrix}
  \psi & \pi_{T_2}(C) \\
  0 & \pi_{F_2}(C)
  \end{bmatrix}$
be a matrix representing a map $\varphi \in \Mor(P_1, P_2)$.
Then $\varphi(x_i) = y_i$ for all $1 \le i \le r_1$ if and only if the following conditions hold:
\begin{enumerate}[label=(\roman*), font=\normalfont, wide=0pt]
\item $\pi_{F_2}(C) = \pi_{F_2}(B) \cdot \pi_{F_1}(A)^{-1}$ is defined over $\Z$
\item $\pi_{T_2}(C) \cdot \pi_{F_1}(A) = \pi_{T_2}(B) - \psi(\pi_{T_1}(A))$.
\end{enumerate}
\end{prop}

\begin{proof}
Observe that $\varphi(x_i) = y_i$ for all $1 \le i \le r_1$ is equivalent to $MA = B$, i.e.
\[
\begin{bmatrix}
  \pi_{T_2}(B) \\
  \pi_{F_2}(B)
  \end{bmatrix} = 
\begin{bmatrix}
  \psi & \pi_{T_2}(C) \\
  0 & \pi_{F_2}(C)
  \end{bmatrix}
\begin{bmatrix}
  \pi_{T_1}(A) \\
  \pi_{F_1}(A)
  \end{bmatrix} = 
\begin{bmatrix}
  \psi(\pi_{T_1}(A)) + \pi_{T_2}(C) \cdot \pi_{F_1}(A) \\
  \pi_{F_2}(C) \cdot \pi_{F_1}(A)
  \end{bmatrix}.
\qedhere \]
\end{proof}

It follows from \Cref{prop:assembly} that $\pi_{F_2}(C)$ is uniquely determined by $A$ and $B$, once well-defined (i.e. exists over $\Z$ -- note that $\pi_{F_1}(A)$ always has an inverse over $\Q$, as it represents an embedding of free abelian groups of rank $r_1$).
On the other hand, $\pi_{T_2}(C)$ need not be uniquely determined by $A$ and $B$, and to obtain all possibilities for $\pi_{T_2}(C)$, we proceed row by row.
Let
\[
P_2^\times := \Z/a_1\Z \oplus \ldots \oplus \Z/a_{n_2}\Z \oplus \Z^{r_2}
\]
be the invariant factor decomposition of $P_2^\times$.
As usual, we express $P_2^\times$ in Smith normal form as the cokernel of a diagonal matrix $\diag(a_1, \ldots, a_{n_2}, 0, \ldots, 0)$, so that the rows of $\pi_{T_2}(B)$ (and $\pi_{T_2}(C)$) are viewed over $\Z/a_1\Z, \ldots, \Z/a_{n_2}\Z$ respectively.

Now fix $1 \le i \le n_2$.
Then by \Cref{prop:assembly}, $[C]_i = [\pi_{T_2}(C)]_i$ must satisfy $[C]_i \cdot \pi_{F_1}(A) = [\pi_{T_2}(B) - \psi(\pi_{T_1}(A))]_i$ as row vectors over $\Z/a_i \Z$ (where $[\cdot]_i$ denotes the $i^\text{th}$ row of a matrix).
Thus we may use (MODSOLVE) to find a solution $[C]_i$, if one exists.
Note however that we only assume that (MODSOLVE) returns a \emph{single} solution, although in actuality one needs to find \emph{all} solutions (which correspond precisely to all possibilities for $[\pi_{T_2}(C)]_i$).

One potential remedy to this would be to assume access to an improved version of (MODSOLVE), which finds all solutions.
However, we take a different approach: setting $Q := F_1/H$, which is a torsion (finite) abelian group, we have a short exact sequence of abelian groups
\[
    0 \to H \xrightarrow{\pi_{F_1}(A)} F_1 \xrightarrow{p} Q \to 0
\]
and applying $\Hom(\cdot, P_2^\times)$ gives a long exact sequence
\begin{equation} \label{eq:LES}
    0 \to \Hom(Q,P_2^\times) \to \Hom(F_1, P_2^\times) \xrightarrow{\alpha} \Hom(H, P_2^\times) \to \Ext^1(Q,P_2^\times) \to \cdots
\end{equation}
In this setting, the goal is to find all lifts of $B \in \Hom(H, P_2^\times)$ via $\alpha$ (note that this provides another proof of \Cref{prop:assembly}(i), since $\alpha$ is given by composition with $\pi_{F_1}(A)$, i.e. $\alpha(\widetilde{f}) = \widetilde{f} \cdot \pi_{F_1}(A)$).
The reasoning above shows that we may use (MODSOLVE) to find a single lift of $B$ (or conclude that none exist).
Then by the exact sequence (\ref{eq:LES}), any two lifts of $B$ differ by an element of $\ker \alpha$, which is the (isomorphic) image of $\Hom(Q, P_2^\times)$ in $\Hom(F_1, P_2^\times)$. 
Since $\Hom(Q, P_2^\times) = \Hom(Q, T_2)$ can in turn be computed by an application of (HOM), this provides a method of determining all lifts of $B$.
Finally, we take as choices for $C$ the lifts of $B$ which send all type 4 pairs at level $r$ to fundamental pairs in $P_2$, which are valid pasture morphisms by \Cref{lem:Hom} and \Cref{rem:alg1Remarks}(iii).

\subsection{Implementation} \label{ssec:options}
Our algorithm for computing pasture morphisms (being comprised of \Cref{alg:fullRankSublattice,alg:depthFirstSearch,ssec:assembly}) has been implemented in Macaulay2 \cite{M2}, due to its capabilities with linear algebra over $\Z$ (e.g. routines for (HOM) and (MODSOLVE)), and the existing \texttt{Matroids} package \cite{chen2019matroids}.
We now list some useful modifications to our algorithm.

First, one is often interested in knowing whether a given matroid $M$ is representable over a specific pasture $P$.
By \Cref{thm:UP}, this is equivalent to the statement $\Mor(F_M, P) \ne \emptyset$.
Thus, it is useful to be able to solve the decision problem of whether or not the set of morphisms is empty (as opposed to computing the entire set).
To do this, one can simply terminate the depth-first search in \Cref{ssec:DFS} as soon as a single candidate is found to be a valid morphism (as in \Cref{ssec:assembly}).
This has been implemented in our Macaulay2 code, and is specified by an optional argument \texttt{FindOne => true}.

Another task of interest is to find isomorphisms between two pastures.
By \Cref{lem:iso}, it suffices to select morphisms whose corresponding group homomorphism in $\Hom(P_1^\times, P_2^\times)$ is given by an invertible matrix.
This can be ensured by imposing suitable rank conditions for candidates in the depth-first search (as well as invertibility of the lifts found in \Cref{ssec:assembly}), i.e. only candidates which increase the rank of the partially-specified morphism are added at each level.
These rank conditions are efficient to check, and have the benefit of reducing the number of candidates even further, sometimes greatly so.
In our Macaulay2 implementation, this is specified by an optional argument \texttt{FindIso => true}.

Finally, one can combine the two modifications above, which solves the decision problem of whether two pastures are isomorphic.
This is particularly useful for classifying matroids via their foundations.

\subsection{Efficiency} \label{ssec:efficiency}

We now discuss the efficiency of \Cref{alg:fullRankSublattice,alg:depthFirstSearch}.
Let $P_1$ and $P_2$ be pastures, and suppose $P_2$ is slim (cf. \Cref{def:pastureConditions}).
For $i \in \{1, 2, 3, 4\}$, set $p_i$ to be the number of type $i$ pairs in the full rank sublattice of $P_1$ as found by \Cref{alg:fullRankSublattice}, so that $r_1 = \rk(P_1^\times) = p_1 + p_2 + 2p_3$.
Then the number of candidates tested in \Cref{alg:depthFirstSearch} is approximately $|\FE(P_2)|^{p_3}$.
Note that $p_3$ equals the number of rule pairs which are $\{ 0, 0 \}$ -- for each such rule pair, there are $|\FE(P_2)|$ choices of candidate image in $P_2$, or equivalently nodes in the corresponding level of the search tree.
The slimness condition on $P_2$ implies that for every other rule pair, there is a (essentially) unique choice of candidate image for the associated fundamental element in $P_1$.
Moreover, the presence of type 4 pairs means that the actual number of candidates tested may be even smaller.

The previous paragraph justifies why we seek to minimize the number of type 3 pairs in \Cref{alg:fullRankSublattice}.
Although it is sometimes possible to create a full rank sublattice by only choosing type 3 pairs at each step, such a sublattice would lead to testing approximately $|\FE(P_2)|^{r_1/2}$ candidates in \Cref{alg:depthFirstSearch}.
Thus we choose a type 3 pair only when there is no other way to increase the rank of the sublattice.

In practice, the efficiency of \Cref{alg:depthFirstSearch} is greatly aided by the following phenomenon: for the foundations of many interesting matroids, \Cref{alg:fullRankSublattice} returns a value of $p_3$ which is small relative to $r_1/2$, cf. \Cref{table:p1p3}.
For instance, the Pappus matroid $P$ has a foundation with multiplicative group of rank $7$, but \Cref{alg:fullRankSublattice} produces only two type 3 pairs, i.e. $p_3 = 2$ (and $p_1 = 3$).
Thus to find all rescaling classes of representations of $P$ over $\F_8$, \Cref{alg:depthFirstSearch} only needs to check $\approx 6^2$ candidates (cf. \Cref{rem:naiveComp}) -- as it turns out, there are $18$ such representations.

\begin{table}[h]
\centering
\caption{Numbers of type $1, 3$ pairs for various matroids}
\begin{tabular}{|c|c|c|c|}
\hline
Matroid & $p_1$ & $p_3$ & $r_1$ \\ \hline
$U(3,6)$ & 6 & 4 & 14 \\
$U(3,7)$ & 16 & 6 & 28 \\
Vamos & 20 & 0 & 20 \\
Pappus & 3 & 2 & 7 \\
Non-Pappus & 8 & 0 & 8 \\ \hline
\end{tabular}
\label{table:p1p3}
\end{table}

Another case of interest is when there are no representations, i.e. $\Mor(F_M, P_2) = \emptyset$.
One might expect this to be the worst case for the decision problem above, i.e. for \Cref{alg:depthFirstSearch} to be slowest in this situation, needing to check many spurious candidates before terminating.
However, this is not always the case.
For instance, when $P_2$ is a field, it turns out that many well-known non-representable matroids (including Vamos, non-Pappus, non-Desargues, and $F_7 \oplus F_7^-$) have foundations such that $1_{F_M}$ is a fundamental element (notice: this is already a theoretical guarantee for non-existence of representations over a field, since any pasture morphism sends $1_{F_M}$ to $1_{P_2}$, and $1$ is never a fundamental element in a field).
Now in \Cref{alg:fullRankSublattice}, a fundamental pair containing $1_{F_M}$ will always be chosen as a type $1$ pair prior to any type 3 pair, and subsequently if $1_{P_2} \not \in \FE(P_2)$, then \Cref{alg:depthFirstSearch} will terminate without needing to check any candidates, having eliminated all branches at the start.

Practically speaking, our software implementation enables computation of representations of matroids on $< 10$ elements, over finite fields of order $< 100$, within a few ($< 5$) minutes on a typical laptop (on average; most matroids which are ``far from uniform'', i.e. have many nonbases, have ``small'' foundations which finish within seconds).
Of course, advances in computational ability bring forth a number of theoretical questions in response, and we expect that further inquiry will result in fruitful research to come.

\section{Applications} \label{sec:matrixReps}

We now mention some useful applications of our algorithms.
First, we can efficiently determine orientability of many small matroids, by detecting existence of a pasture morphism from the foundation to the sign hyperfield $\mathbb{S}$.
Since $\mathbb{S}$ is a small pasture (note $|\FE(\mathbb{S})| = 2$), \Cref{alg:depthFirstSearch} for maps to $\mathbb{S}$ with \texttt{FindOne => true} is quick (see \Cref{ssec:options,ssec:efficiency}).
This effectively reduces the runtime of determining orientability to that of computing (a full rank sublattice of) the foundation, i.e. \Cref{alg:TutteGroup,alg:innerTutteGroup,alg:additiveRelations,alg:fullRankSublattice}.
For example, in \cite[Proposition 5.8]{Robbins2023} the authors study the matroid $NC_{9,1}$ without addressing the question of its orientability.
With our software, we can instantly ($< 1$ second) verify that $NC_{9,1}$ is not orientable.

Next, we consider criteria for non-representability of a matroid over any field.
Although it is known (cf. \cite{nelson2016almost}) that asymptotically all matroids are non-representable, concrete examples of non-representable matroids are rare (perhaps owing to the difficulty of proving that a specific matroid is non-representable).
In \Cref{ssec:efficiency} we discussed one sufficient criterion for non-representability: namely $1 \in \FE(F_M)$.
However, this condition is not necessary: there are exactly 4 matroids on $\le 9$ elements (called $R_9^A, R_9^B$, and their duals) which do not have $1$ as a fundamental element, but are not representable over any field (see \cite{sage}).
We may generalize and explain these examples as follows: consider the pasture 
\[
P_0 := \F_1^\pm \langle x,y,z,w\rangle \sslash \{x+y-1,x+z-1, \frac{y}{z}+w-1\}
\]
which has $P_0^\times = \langle \FE(P_0) \rangle \cong \Z/2\Z \oplus \Z^4$, and 3 hexagons of type $\U$, given by $3$ fundamental pairs $(x, y), (x, z), (y/z, w)$.
Since fields are slim, and $1$ is not a fundamental element in any field, one immediately verifies that there are no pasture morphisms $P_0 \to k$ for any field $k$.
Thus to show that a matroid $M$ is non-representable, it suffices to exhibit a single pasture morphism from $P_0$ to $F_M$, i.e. show that $\Mor(P_0, F_M) \ne \emptyset$.
Using our software, this is quick to check for $R_9^A$ and $R_9^B$.
In addition, if $P_2$ is any pasture with $1 \in \FE(P_2)$, i.e. $( a, 1 ) \in \FP(P_2)$ for some $a \in P_2^\times$, then there is a pasture morphism $P_0 \to P_2$ sending $x, w \mapsto a$ and $y, z \mapsto 1$.
(Note that this gives a simple proof of non-representability for all the non-representable matroids mentioned thus far, and demonstrates an advantage of working in the category of pastures.)
We therefore pose the following question:

\begin{question} \label{q:nonrep}
Is there a matroid $M$ such that $M$ is not representable over any field, but $\Mor(P_0, F_M) = \emptyset$?
\end{question}

We expect that most non-representable matroids have $1$ as a fundamental element, so \Cref{q:nonrep} is likely to be ``generically'' false.
Moreover, a smallest matroid satisfying \Cref{q:nonrep}, if it exists, would necessarily have $\ge 10$ elements.

Finally, we discuss how to compute matrix representations of matroids over finite fields.
By \Cref{thm:UP}, a morphism from the foundation of a matroid $M$ to a pasture $P$ corresponds to a rescaling equivalence class of a $P$-representation of $M$.
We now wish to realize each equivalence class by a concrete matrix.
As mentioned in \Cref{ex:rescalingEquivalenceClassAndIsomorphismClass}, for a fixed basis $B_0$ of $M$, there is a unique reduced row echelon form with respect to $B_0$ for the isomorphism class of any $P$-representation of $M$.
Hence it suffices to extend morphisms in $\Mor(F_M, P)$ to morphisms in $\Mor(P_M, P)$ (where $P_M$ is the universal pasture, see \Cref{thm:UP}).
Given $f\in \Mor(F_M,P)$, we can obtain an extension $\hat{f}$ by finding a splitting $p$ of the exact sequence:

\[
\begin{tikzcd} [column sep=4em, row sep=4em]
0 \arrow{r} & F_M^\times \arrow{r}{i} \arrow{d}{f} & P_M^\times \cong \T_M \arrow[]{r}{\deg}\arrow[dashrightarrow, bend left=22]{l}{p} \arrow[dashrightarrow]{ld}{\hat{f}} & \Z^n \\%
& P^\times
\end{tikzcd}
\]
Moreover, by \cite[Theorem 7.25]{BAKER2021107883}, the rescaling equivalence class of the Grassmann-Pl\"{u}cker function is independent of the choice of the extension, i.e. different splittings may give different Grassmann-Pl\"{u}cker functions, but these Grassmann-Pl\"{u}cker functions must be in the same rescaling equivalence class.
Thus we can use the projection from \Cref{alg:innerTutteGroup} to obtain a Grassmann-Pl\"{u}cker function $\varphi := f\circ \rho^{(0)}$.
\[ \begin{tikzcd} [column sep=4em, row sep=4em]
   G \arrow{r}{\rho} \arrow[bend right=11, swap]{rr}{\rho^{(0)}} \arrow[dashrightarrow]{rrd}{\varphi} & P_M^\times \cong \mathbb T_M \arrow[]{r}{} & F_M^\times \cong \T^{(0)}_M \arrow{d}{f} \\%
& &  P^\times
\end{tikzcd}
\label{eq:splitting} \tag{*}
\]
We illustrate the full procedure of computing representations of a matroid over a finite field via an example.

\begin{ex} \label{ex:mainEx}
\begin{figure}[h]
    \centering
    \begin{tikzpicture}
\draw[thick] (-3,0) -- (3,0);
\draw[thick] (-3,0) -- (1.5,2);
\draw[thick] (-3,0) -- (1.5,-2);
\draw[thick] (-0.75,-1) -- (1.5,2);
\draw[thick] (-0.75,1) -- (1.5,-2);
\filldraw (-3,0) circle (3pt)
node[above]{0};
\filldraw (0,0) circle (3pt)
node[above]{5};
\filldraw (3,0) circle (3pt)
node[above]{6};
\filldraw (1.5,2) circle (3pt)
node[below]{4};
\filldraw (1.5,-2) circle (3pt)
node[above]{2};
\filldraw (-0.75,-1) circle (3pt)
node[below]{1};
\filldraw (-0.75,1) circle (3pt)
node[above]{3};
\end{tikzpicture}
    \caption{Matroid $M$}
    \label{fig:ex22}
\end{figure}
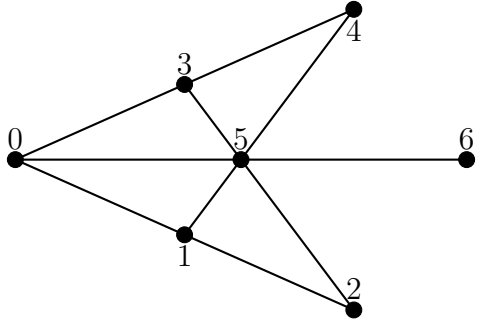

Let $M$ be the matroid shown in \autoref{fig:ex22} (of rank $3$ on $7$ elements, with 5 nonbases depicted as 3-point lines).

Applying \Cref{alg:innerTutteGroup,alg:TutteGroup,alg:additiveRelations} 
with the fixed basis $B_0 := \{0,1,3\}$ of $M$
yields the presentation $F_M^\times = \Z/2\Z \oplus \Z^3 = \coker(\diag(2,0,0,0))$, with torsion subgroup generated by $\epsilon = \begin{bmatrix} 1 & 0 & 0 & 0 \end{bmatrix}^T$, and 3 hexagons of type $\U$, given by the following 3 fundamental pairs:
\[
( u_1, u_2 ) := \left(
\begin{bmatrix}
0 \\
0 \\
0 \\
1 \\
\end{bmatrix}
,
\begin{bmatrix}
1 \\
0 \\
1 \\
0 \\
\end{bmatrix}
\right),
( v_1, v_2 ) := \left(
\begin{bmatrix}
0 \\
1 \\
-1 \\
0 \\
\end{bmatrix}
,
\begin{bmatrix}
1 \\
0 \\
-1 \\
1 \\
\end{bmatrix}
\right),
( w_1, w_2 ) := \left(
\begin{bmatrix}
0 \\
-1 \\
0 \\
2 \\
\end{bmatrix}
,
\begin{bmatrix}
1 \\
-1 \\
2 \\
0 \\
\end{bmatrix}
\right).
\]
The full rank sublattice chosen by \Cref{alg:fullRankSublattice} is
\[
\widetilde{H} := \langle u_1, u_2, v_1 \rangle.
\]

Consider the problem of finding representations of $M$ over $\F_5$.
We fix a primitive element $2 \in \F_5^\times$, which facilitates conversion between additive and multiplicative notation, via $\Z/4\Z \xrightarrow{\sim} \F_5^\times$, $a \mapsto 2^a$.
Then as a pasture, $\F_5$ has multiplicative group $\Z/4\Z = \coker(\begin{bmatrix} 4 \end{bmatrix})$, along with $\epsilon = 2$ (since $-1 = 4 = 2^2$), and 1 hexagon $( ( 1, 2 ), ( -1, -1 ), ( 2, 1 ) )$ of type $\D$.

Now in $\widetilde{H}$: we have that $( u_1, u_2 )$ is a type 3 pair, and $v_1$ belongs to a type 1 pair wrt $\langle u_1, u_2 \rangle$, since $v_2 = u_1 - u_2$.
Per \Cref{alg:depthFirstSearch}, we run over all fundamental pairs $( 1, 2 ), ( -1, -1 ), ( 2, 1 )$ in $\F_5$, which are all possible images of $( u_1, u_2 )$.
For each case, we determine the possible images of $v_2$ (the partner of $v_1$), which are $-1$ $(= 1 - 2)$, 0 $(= -1 - (-1))$, and 1 $(= 2 - 1)$ respectively.
However, since $v_2 \in \FE(F_M)$, its image must be in $\FE(\F_5)$, and thus cannot be $0$ (as $2^0 = 1 \not \in \FE(\F_5)$).
This eliminates $( -1, -1 )$ as a possible image of $( u_1, u_2 )$.
For the other 2 cases: if $v_2 \mapsto -1$, then $v_1$ must be mapped to a partner of $-1$ in $\F_5$, which must be $-1$; similarly if $v_2 \mapsto 1$, then $v_1 \mapsto 2$.

Next, we assemble the candidates as in \Cref{ssec:assembly} to obtain two pasture morphisms:
\begin{equation} \label{eq:morphisms}
\left\{
\begin{bmatrix}
2&-1&0&1
\end{bmatrix},
\begin{bmatrix}
2&1&-1&2
\end{bmatrix}
\right\}.
\end{equation}
From \Cref{alg:TutteGroup,alg:innerTutteGroup}, we have a map $\rho^{(0)} : G \twoheadrightarrow F_M$ as in (\ref{eq:splitting}), which is a $4 \times 31$ matrix
\setcounter{MaxMatrixCols}{35}
\addtolength{\arraycolsep}{-1.5pt}
\[
\begin{bmatrix}
-1&0&0&0&0&0&0&0&0&0&0&0&0&0&0&1&1&-1&-1&-1&-1&0&0&-1&0&0&0&0&0&0&0\\
0&0&0&0&0&0&0&0&0&0&0&0&0&0&0&0&0&0&0&0&0&0&0&0&0&1&0&0&0&0&0\\
0&0&0&0&0&0&0&0&0&0&0&0&0&0&0&0&0&1&1&1&1&0&0&1&0&0&1&0&0&0&0\\
0&0&0&0&0&0&0&0&0&0&0&0&0&0&0&0&0&0&0&0&0&0&1&0&1&0&0&1&1&1&1
\end{bmatrix}
\]
Composing this with e.g. the first pasture morphism in (\ref{eq:morphisms}) gives the following vector recording the Grassmann-Pl\"ucker coordinate of each basis of $M$:
\[
\begin{bmatrix}
-2&0&0&0&0&0&0&0&0&0&0&0&0&0&0&2&2&-2&-2&-2&-2&0&1&-2&1&-1&0&1&1&1&1
\end{bmatrix}
\]
\addtolength{\arraycolsep}{1.5pt}
Finally, we generate a matrix $A = (a_{i,j})_{0 \le i < 3, 0 \le j < 7}$ from the Grassmann-Pl\"{u}cker coordinates.
First, fix a bijection $g : \{ 0,1,2 \} \xrightarrow{\sim} B_0$.
We take the submatrix of $A$ with columns in $B_0$ to be a $3 \times 3$ identity matrix.
For the remaining entries, we set $a_{i,j} = 0$ if $B_0 \backslash \{g(i)\} \cup \{j\}$ is not a basis; otherwise we set $a_{i,j} = 2^b$, where $b$ is the Grassmann-Pl\"ucker coordinate of $B_0 \backslash \{g(i)\} \cup \{j\}$.
Applying this process to the morphisms in (\ref{eq:morphisms}) gives the two matrices
\begin{equation} \label{eq:reps}
\left\{
\begin{bmatrix}
1&0&1&0&1&1&1\\ 
0&1&1&0&0&1&-1\\
0&0&0&1&1&1&-1
\end{bmatrix},
\begin{bmatrix}
1&0&1&0&1&1&1\\
0&1&1&0&0&1&2\\
0&0&0&1&1&1&2
\end{bmatrix}
\right\}.
\end{equation}
These are exactly the 2 (distinct up to rescaling) representations of $M$ over $\F_5$.
\end{ex}

Below is a Macaulay2 session reproducing \Cref{ex:mainEx}:
\begin{verbatim}
Macaulay2, version 1.22

i1 : load "Matroids/foundations.m2"

i2 : M = matroid({{0,1,2},{0,3,4},{0,5,6},{1,4,5},{2,3,5}}, EntryMode => "non
bases")

o2 = a "matroid" of rank 3 on 7 elements

i3 : elapsedTime foundation M
 -- 0.132819 seconds elapsed

o3 = a "foundation" on 4 generators with 3 hexagons

i4 : k = GF 5; peek pasture k

o5 = Pasture{cache => CacheTable{}                                           }
             epsilon => | 2 |
             hexagons => {{{| 1 |, | 2 |}, {| -1 |, | -1 |}, {| 2 |, | 1 |}}}
             multiplicativeGroup => cokernel | 4 |

i6 : elapsedTime morphisms(foundation M, pasture k)
 -- 0.0759619 seconds elapsed

o6 = {| 2 -1 0 1 |, | 2 1 -1 2 |}

i7 : representations(M, k)

o7 = {| 1 0 1 0 1 1 1  |, | 1 0 1 0 1 1 1 |}
      | 0 1 1 0 0 1 -1 |  | 0 1 1 0 0 1 2 |
      | 0 0 0 1 1 1 -1 |  | 0 0 0 1 1 1 2 |

i8 : all(o7, A -> matroid A == M)

o8 = true

\end{verbatim}




\textbf{Acknowledgments.} We are first and foremost indebted to Matt Baker and Oliver Lorscheid, for introducing us to this topic and continual support throughout the project -- this work would not be possible without them.
We would also like to thank Rudi Pendavingh for generously sharing unpublished Sage code for computing Tutte groups, which directly inspired \Cref{alg:TutteGroup}.

\begin{small}
\bibliographystyle{plain}
\bibliography{mybib}

\begin{thebibliography}{10}

\bibitem{Baker-Lorscheid20}
Matthew Baker and Oliver Lorscheid.
\newblock Foundations of matroids. {P}art 1: Matroids without large uniform
  minors.
\newblock Preprint, \arxiv{2008.00014}, 2020.

\bibitem{Baker-Lorscheid18}
Matthew Baker and Oliver Lorscheid.
\newblock Lift theorems for representations of matroids over pastures.
\newblock Preprint, \arxiv{2107.00981}, 2021.

\bibitem{BAKER2021107883}
Matthew Baker and Oliver Lorscheid.
\newblock The moduli space of matroids.
\newblock {\em Advances in Mathematics}, 390:107883, 2021.

\bibitem{brualdi1969comments}
Richard~A Brualdi.
\newblock Comments on bases in dependence structures.
\newblock {\em Bulletin of the Australian Mathematical Society}, 1(2):161--167,
  1969.

\bibitem{chen2019matroids}
Justin Chen.
\newblock Matroids: a \text{Macaulay2} package.
\newblock {\em Journal of Software for Algebra and Geometry}, 9(1):19--27,
  2019.

\bibitem{creech2021limits}
Steven Creech.
\newblock Limits and colimits in the category of pastures.
\newblock Preprint, \arxiv{2103.08655}, 2021.

\bibitem{Dress1989GeometricAF}
Andreas W.~M. Dress and Walter Wenzel.
\newblock Geometric algebra for combinatorial geometries.
\newblock {\em Advances in Mathematics}, 77:1--36, 1989.

\bibitem{M2}
Daniel~R. Grayson and Michael~E. Stillman.
\newblock Macaulay2, a software system for research in algebraic geometry.
\newblock Available at \url{http://www.math.uiuc.edu/Macaulay2/}.

\bibitem{lunelli2002representation}
Massimiliano Lunelli and Antonio Laface.
\newblock Representation of matroids.
\newblock Preprint, \arxiv{math/0202294}, 2002.

\bibitem{nelson2016almost}
Peter Nelson.
\newblock Almost all matroids are non-representable.
\newblock Preprint, \arxiv{1605.04288}, 2016.

\bibitem{Oxley92}
James~G. Oxley.
\newblock {\em Matroid theory}.
\newblock Oxford Science Publications. The Clarendon Press, Oxford University
  Press, New York, 1992.

\bibitem{Robbins2023}
Jakayla Robbins and Daniel Slilaty.
\newblock Orientations of golden-mean matroids.
\newblock {\em Journal of Combinatorial Theory, Series B}, 162:71--117, 2023.

\bibitem{sage}
{Sage}.
\newblock Documentation for the matroids in the catalog.
\newblock Available at \\
  \url{https://doc.sagemath.org/html/en/reference/matroids/sage/matroids/catalog.html}.

\end{thebibliography}
\end{small}

\end{document}